\documentclass[a4paper,12pt]{amsart}

\usepackage{amsmath,amssymb,amsthm,amsfonts}
\usepackage{color}
\usepackage[arrow, matrix, curve]{xy}
\usepackage{tikz}
\usepackage{graphicx}
\usepackage{geometry}
\usepackage{enumerate}
\usepackage{hyperref}

\newcommand{\ZZ}{\mathbb{Z}}
\newcommand{\QQ}{\mathbb{Q}}
\newcommand{\PP}{\mathbb{P}}
\newcommand{\CC}{\mathbb{C}}
\newcommand{\FF}{\mathcal{F}}
\newcommand{\OO}{\mathcal{O}}

\newcommand{\Ex}{\mathcal{E}}
\newcommand{\Ts}{\mathcal{A}}
\newcommand{\Db}{\mathcal{D}^b}
\newcommand{\mns}{{\scalebox{0.6}[1.0]{$-$}}}

\DeclareMathOperator{\rk}{rk}

\DeclareMathOperator{\coker}{coker}
\DeclareMathOperator{\Hom}{Hom}

\DeclareMathOperator{\Ext}{Ext}
\DeclareMathOperator{\TV}{TV}
\DeclareMathOperator{\Pic}{Pic}
\DeclareMathOperator{\Aug}{Aug}
\DeclareMathOperator{\ch}{ch}

\theoremstyle{theorem}
\newtheorem{thm}{Theorem}[section]

\newtheorem*{thm*}{Theorem}

\newtheorem{prop}[thm]{Proposition}
\newtheorem{lem}[thm]{Lemma}
\newtheorem*{conj}{Conjecture}
\theoremstyle{definition}
\newtheorem{defn}[thm]{Definition}
\newtheorem{rmk}[thm]{Remark}
\newtheorem{ex}[thm]{Example}
\newtheorem*{ex*}{Example}

\newcommand{\vimende}{\color{white} \\ \color{gray} \thispagestyle{empty} \vfill\noindent\texttt{:wq}}

\newcommand{\tred}[1]{\textcolor{red}{#1}}
\newcommand{\tblue}[1]{\textcolor{blue}{#1}}

\author{Andreas Hochenegger}
\address{Institut f\"ur Mathematik und Informatik,
         Freie Universit\"at Berlin,
         Arnimallee 3,
         14195 Berlin, Germany}
\email{hochen@math.fu-berlin.de \vimende}

\usetikzlibrary{patterns}
\tikzstyle{mylines}=[thick]

\newcommand{\figHirzeFan}{
\begin{tikzpicture}
\draw[thick] (2.5,2.5) -- (0.5,2.5) node[anchor=south]{$(\mns1,0)$};
\draw[thick] (2.5,0.5)  node[anchor=west]{$(0,\mns1)$} -- (2.5,4.5) node[anchor=east]{$(0,1)$};
\draw[thick] (2.5,2.5) -- (3.25,4.25) node[anchor=west]{$(1,r)$};
\end{tikzpicture}
}

\newcommand{\figHirzeIntersect}{
\begin{tikzpicture}
\draw[thick] (0.5,4) .. controls (2.5,3.5) .. node[anchor=south]{$\mns r$} (4,4);
\draw[thick] (0.5,4.5) .. controls (1.4,2.5) .. node[anchor=east]{$0$} (0.5,0.5);
\draw[thick] (0.5,1) .. controls (2.5,1.5) .. node[anchor=north]{$r$} (4.5,1);
\draw[thick] (4.75,0.5) .. controls (3,2.5) .. node[anchor=west]{$0$} (3.5,4.5);
\fill[pattern=north west lines] (0.9,1.2) .. controls (1.4,2.6) .. (0.9,3.8) .. controls (2.5,3.4) .. (3.2,3.6) .. controls (3,2.5) .. (4,1.2) .. controls (2.5,1.6) .. (0.9,1.2);
\fill[color=white] (2,2.2) rectangle (2.6,2.8);
\draw (2.3,2.5) node {$T$};
\end{tikzpicture}
}

\newcommand{\figBlowdownX}{
\begin{tikzpicture}
\begin{scope}[xshift=0cm,yshift=4cm]
\draw (0,1) node [anchor=east] {\tiny$D_4$} -- (1,1);
\draw[color=blue] (1,1) -- (2,1) node [anchor=west] {\tiny$D_1$};
\draw (0.1,0.1) node [anchor=north east] {\tiny$D_5$} -- (1,1) -- (1.9,1.9) node [anchor = south west] {\tiny$D_2$};
\draw[color=blue] (1,0) node [anchor=north] {\tiny$D_6$} --  (1,1);
\draw (1,1) -- (1,2) node [anchor=south] {\tiny$D_3$};
\draw[color=red] (1.9,0.1) node [anchor=north west] {\tiny$D_7$} -- (1,1);
\draw (1,-0.5) node[anchor=north] {\tiny $X = X_4 = \TV(\tblue{\mns 2},\mns 1,\mns 1,\mns 1,\mns 1,\tblue{\mns 2},\tred{\mns 1})$};
\end{scope}
\draw[thick] [->] (3,5) -- (4,5);
\begin{scope}[xshift=5cm,yshift=4cm]
\draw[color=blue] (0,1) -- (1,1);
\draw (1,1) -- (2,1);
\draw[color=red] (0.1,0.1) -- (1,1);
\draw (1,1)  -- (1.9,1.9);
\draw[color=blue] (1,0) --  (1,1);
\draw (1,1) -- (1,2);
\draw (1,-0.5) node[anchor=north] {\tiny $X_3 = \TV(\mns 1,\mns 1,\mns 1,\tblue{\mns 1},\tred{\mns 1},\tblue{\mns 1})$};
\end{scope}
\draw[thick] [->] (6,3) -- (6,2.25);
\begin{scope}[xshift=5cm,yshift=0cm]
\draw (0,1) -- (1,1);
\draw[color=red] (1,1) -- (2,1);
\draw[color=blue] (1,1)  -- (1.9,1.9);
\draw[color=blue] (1,0) --  (1,1);
\draw (1,1) -- (1,2);
\draw (1,-0.5) node[anchor=north] {\tiny $X_2 = \TV(\tred{\mns 1},\tblue{\mns 1},\mns 1, 0,\tblue{0})$};
\end{scope}
\draw[thick] [->] (4,1) -- (3,1);
\begin{scope}[xshift=0cm,yshift=0cm]
\draw (0,1) -- (1,1);
\draw (1,1)  -- (1.9,1.9);
\draw (1,0) --  (1,1);
\draw (1,1) -- (1,2);
\draw (1,-0.5) node[anchor=north] {\tiny $\FF_1 = X_1 = \TV(0,\mns 1,0,1)$};
\end{scope}
\end{tikzpicture}
}

\newcommand{\figAutoX}{
\begin{tikzpicture}
\draw[dashed] (0,2) -- (2.2,-0.2);
\draw [<->] (0,1.25) .. controls (0.2,1.8) .. (0.75,2) node[anchor=south] {\tiny$f$};
\draw[thick] (0,1) -- (1,1);
\draw[thick] (1,1) -- (2,1) node [anchor=west] {\tiny$D_1$};
\draw[thick] (0.1,0.1) -- (1,1) -- (1.9,1.9);
\draw[thick] (1,0) node [anchor=north] {\tiny$D_6$} --  (1,1);
\draw[thick] (1,1) -- (1,2);
\draw[thick] (1.9,0.1) -- (1,1);
\end{tikzpicture}
}

\title[Exceptional Sequences and Spherical Twists]{Exceptional Sequences of Line Bundles and Spherical Twists -- a Toric Example}

\begin{document}

\maketitle

\begin{abstract}
Exceptional sequences of line bundles on a smooth projective toric surface are automatically full when they can be constructed via augmentation. By using spherical twists, we give examples that there are also exceptional sequences which can not be constructed this way but are nevertheless full.
\end{abstract}

\section{Introduction}
Let $X$ be a (smooth projective) toric surface. The complement of the acting torus $T = (\CC^\ast)^2$ forms a circular sequence of the $T$-invariant divisors $D_1, \ldots, D_n$, which satisfy certain intersection conditions and whose sum is the anticanonical divisor $\mns K_X$. A configuration of divisors $\Ts = (A_1, \ldots, A_n)$ which satisfy the same conditions is a so-called \emph{toric system} as introduced by L. Hille and M. Perling in \cite{hille:perling:08}. 
The key point to notice here is that any exceptional sequence of line bundles gives rise to such a toric system, which is then called exceptional, too. Moreover, given a blow-up $X' \to X$ in a point, there is a process called \emph{augmentation} to build toric systems $\Aug_i \Ts$ on $X'$ from a given toric system $\Ts$ on $X$. This process preserves the exceptionality of a toric system.
Finally, we can associate a toric variety $\TV(\Ts)$ to any toric system $\Ts$.

The birational geometry of toric surfaces is rather easy. We can obtain any (smooth projective) toric surface $X$ from a Hirzebruch surface $\FF_{\!r}$ by finitely many blow-ups, as long as $X \not= \PP^2$. Augmentation is quite similar to blow-ups, in that $\TV(\Aug_i \Ts)$ is a blow-up of $\TV(\Ts)$.
Therefore, the question arises whether any exceptional toric system on $X \not= \PP^2$ is \emph{constructible}, i.e. can be obtained from an exceptional toric system on $\FF_{\!r}$ by finitely many augmentations.
In \cite{me:ilten:11b}, N. Ilten and the author showed that this is true for toric surfaces of Picard rank $3$ and $4$, whereas for Picard rank $5$, they found a non-constructible exceptional toric system.

One subtle detail was omitted there, namely, whether the associated exceptional sequence is \emph{full}, i.e. generates the bounded derived category of coherent sheaves on that surface. For constructible toric systems, this is automatically true. The aim of this paper is to show that this example is full by using an autoequivalence of the derived category which maps the non-constructible exceptional sequence to a full one. By \cite{broom:ploog} of N. Broomhead and D. Ploog, only a spherical twist can serve for that purpose.

We give a short outline of the paper. Section~\ref{sec:preliminaries} contains basics on toric surfaces, as well as necessary results on toric systems. The main point is that toric systems with the same associated toric surface differ by special base changes of $\Pic(X)$, called \emph{$K$-isometries}. The group of $K$-isometries is known for low Picard ranks explicitly and is finite. In Section~\ref{sec:sphtwist} we have a closer look at spherical twists. We describe which twists preserve line bundles. Actually, these twists can be considered as lifts of certain $K$-isometries to the derived category. In Section~\ref{sec:nonconstr:exseq} we investigate an orbit of these $K$-isometries on a certain toric surface and discover two non-constructible exceptional sequences. We show that applying a spherical twist to these makes them constructible. As a consequence, they are full.

\subsection*{Acknowledgments} The author would like to thank David Ploog for useful discussions and helpful comments.
Moreover, the author would like to thank Nathan Ilten for carefully reading a preliminary version of this paper.

\section{Preliminaries}
\label{sec:preliminaries}

We will highlight here some facts about toric surfaces mostly to fix notation.
For details, we refer the reader to the ample literature, like \cite{danilov}, \cite{oda}, \cite{fulton:93a} or \cite{cox:little:schenk}.

All toric surfaces in this article will be smooth and projective. 
Let $X$ be such a surface. Then $X$ is $\TV(\Sigma)$ for a fan $\Sigma$ in $N_\QQ = N \otimes \QQ$, where $N \cong \ZZ^2$ is a $2$-dimensional lattice, such that
\begin{itemize}
\item (Smoothness) For each $2$-dimensional cone $\sigma \in \Sigma$, the set of generators $v_i$ of the rays $\rho_i$ of $\sigma$ forms a lattice basis of $N$, and,
\item (Completeness) $|\Sigma| = \bigcup_{\sigma \in \Sigma} \sigma = N_\QQ$.
\end{itemize}
Actually, to get some picture of $X$ we can look at the divisors, which are invariant under the action of the $2$-dimensional torus $T = (\CC^\ast)^2$. To each such prime divisor $D_i$ corresponds exactly one ray $\rho_i$ of the fan $\Sigma$. We will always enumerate the rays $\rho_1, \ldots, \rho_n$ in a cyclic manner. 
With that in mind, the $T$-invariant divisors satisfy the following properties
\begin{itemize}
\item (Intersection) $D_i.D_j = \begin{cases} 1 & \text{if $D_i$ and $D_j$ are adjacent, or,} 
\\ 0 & \text{if they are neither equal nor adjacent.} \end{cases}$
\item (Canonical divisor) $\sum_i D_i = \mns K_X$.
\end{itemize}
\begin{figure}
\figHirzeFan
\hspace{2cm}
\figHirzeIntersect
\caption{The Hirzebruch surface $\FF_{\!r}$.}
\label{fig:hirze}
\end{figure}
It's worth noting here that $X$ is completely determined by the self-intersection numbers $a_i = D_i^2$, since the rays satisfy the relation $v_{i-1} + a_i v_i + v_{i+1} = 0$~\footnote{Here (and subsequently) we consider the index, say $i$, of such cyclically arranged objects as an element of $\ZZ/n\ZZ = \{1,\ldots,n\}$.}. So we will sometimes write $X = \TV(a_1,\ldots,a_n)$.
Another point we want to note here is that $n = \rk K_0(X)$.

\subsection*{Toric Systems}
We review now a generalisation of the above observations about the $T$-invariant divisors as done in \cite{hille:perling:08}, and we refer the reader to this article for details. 

We call a sequence $\Ts=(A_1,\ldots,A_n)$ of divisors on $X$ of length $n= \rk K_0(X)$ a \emph{toric system} if its elements satisfy the conditions (Intersection) and (Canonical divisor) above.
To such a system $\Ts$, it is possible to associate a (smooth, projective) toric surface $\TV(\Ts) := \TV(A_1^2, \ldots, A_n^2)$.  
Obviously, $(D_1, \ldots, D_n)$ forms a toric system which we call the \emph{standard} toric system, and the associated toric variety is $X$ itself.

\begin{ex}[{\cite[Proposition 5.2]{hille:perling:08}}]
\label{ex:hirze:ts}
All toric systems on a Hirzebruch surface $\FF_{\!r} = \TV(r,0,\mns r,0)$ up to cyclic permutation or reflection of the indices are of the form
\begin{align*}
        \Ts_{r,i}&=\big(P, iP+Q, P, \mns (r+i)P+Q\big);\text{ and}\\
        \tilde\Ts_{r,i}&=\big(Q-\frac{r}{2}P,P+i(Q-\frac{r}{2}P),Q-\frac{r}{2}P,P-i(Q- \frac{r}{2}P)\big)\text{ if $r$ is even}.
\end{align*}
Here, we choose $P$ as the class of the fibre of ruling with $P^2 =0$  and $Q^2 = r$ and $P.Q = 1$.

The associated toric systems are
$\TV(\Ts_{r,i})=\FF_{|r+2i|}$ and $\TV(\tilde\Ts_{r,i})=\FF_{|2i|}$.
\end{ex}

There is a more conceptual approach using Gale-duality. We can write down the short exact sequence:
\[
\xymatrix@R=0pt{
0 \ar[r] & \Pic(X) \ar[r] & \ZZ^n \ar[r]^{\coker\qquad} & \ZZ^n/\Pic(X) \cong \ZZ^2 \ar[r] & 0\\
& D \ar@{|->}[r] & (D.A_1, \ldots, D.A_n)
}
\] 
The duality here is between the sequence $(A_1,\ldots,A_n)$ in $\Pic(X)$ and the images $(v_1,\ldots,v_n)$ of the standard basis of $\ZZ^n$ under the projection to the cokernel $\ZZ^2$.
It turns out that these images $(v_1,\ldots,v_n)$ are actually the ray generators of the fan of $\TV(\Ts)$.

There is one important point which becomes essential in the following sections. This duality is only \emph{up to base change}. This means that given two toric systems $\Ts$ and $\Ts'$, the corresponding toric varieties $\TV(\Ts)$ and $\TV(\Ts')$ are equal if and only if these toric systems differ by an automorphism of $\Pic(X)$. Due to the defining properties of toric systems, such an automorphism has to preserve the intersection pairing and the canonical divisor $K_X$. For that reason we call such an automorphism a \emph{$K$-isometry}.
Before we have a closer look at these $K$-isometries, we recall some facts about

\subsection*{The Birational Geometry of Toric Surfaces}

We don't want to go deep into the Minimal Model Theory for toric varieties as presented in \cite{cox:little:schenk} or \cite{matsuki}. We only need well-known facts about the birational geometry of rational surfaces, see for example \cite[Chapter III]{manin}.

The first thing to note is that the blow-up of a toric surface $X$ given as $\TV(a_1, \ldots, a_n)$ in a $T$-invariant point can be described purely combinatorially. Such a point is the intersection of two adjacent $T$-invariant divisors $D_i$ and $D_{i+1}$. The blow-up $\tilde{X}$ is again toric and has the form
\[
\TV(a_1, \ldots,a_{i-1}, a_{i}-1,\mns 1,a_{i+1}-1,a_{i+2} \ldots, a_n).
\]
By Minimal Model Theory, we can blow-down any (smooth, projective) toric surface $X$ in $T$-invariant points finitely many times, until we arrive at a Mori fibration. There are only two such cases, namely $\PP^2 = \TV(1,1,1) \to \{pt\}$ or the Hirzebruch surface $\FF_{\!r}=\TV(r,0,\mns r,0) \to \PP^1$.
So we can reduce many statements about a general toric surface by induction to a statement about $\PP^2$ and $\FF_{\!r}$. 

\begin{rmk}
\label{rmk:exdiv:orth}
Given such a toric blow-up $b\colon X' \to X$ with exceptional divisor $R$,
we obtain an inclusion $b^\ast \colon \Pic(X) \hookrightarrow \Pic(X')$. By this inclusion $\Pic(X) = \OO(R)^\perp$ with respect to the intersection pairing. Here and subsequently, we will abuse notation and omit $b^\ast$ for pulled-back line bundles and divisors.
\end{rmk}

\begin{defn}[{+ Proposition, cf. \cite[Section 2]{hille:perling:08}}]
\label{defn:augment}
Consider a toric system $\Ts = (A_1, \ldots, A_n)$ on $Y$ and a blow-up $b\colon X' \to X$.
Then 
\[
\Ts' = \Aug_i (A_1, \ldots, A_{i-1}, A_{i} - R, R, A_{i+1}-R, A_{i+1}, \ldots, A_n)
\]
is called the \emph{augmentation} of $\Ts$ at position $i$, which is a toric system on $X'$.
Moreover, note that $\TV(\Ts')$ is a toric blow-up of $\TV(\Ts)$ in the toric fixed point which is the intersection of the $i$-th and $(i+1)$-st $T$-invariant divisor on $\TV(\Ts)$.
\end{defn}

Now we return to a closer inspection of

\subsection*{$K$-Isometries}
On a Hirzebruch surface $\FF_{\!r} = \TV(r,0,\mns r,0)$ with $r=2a+1$, there is a special basis $(H,R_1)$ of $\Pic(\FF_{\!r})$ which diagonalises the intersection pairing with signature $(1,\mns 1)$. In terms of the $T$-invariant Divisors we can choose
\[
H = \OO(D_1-aD_2) \text{ and } R_1 = \OO(D_1-(a+1)D_2).
\]
Any toric surface $X$ with $\rho(X)>2$ can be blown down to such a Hirzebruch surface, say by $X=X_l \to \cdots \to X_1 = \FF_{2a+1}$. Inductively, we obtain a similar basis $(H,R_1,\ldots,R_l)$ which diagonalises the pairing with signature $(1,\mns 1,\ldots,\mns 1)$, namely, by pulling back the according basis of $\Pic(X_{l-1})$ and extending by the line bundle $R_l$ of the exceptional divisor.
We call such a basis a \emph{good basis} of $\Pic(X)$.

\begin{prop}[{cf. \cite[Theorem 23.9(ii)]{manin}}]
\label{prop:kiso:weyl}
Let $X$ be a toric surface of Picard rank $3\leq \rho(X) \leq9$ with good basis $(H,R_1,\ldots,R_l)$.
Then the following groups coincide
\begin{itemize}
\item the group of $K$-isometries of $\Pic(X)$, and,
\item the Weyl group $W_X$ of the roots 
\[
\Phi_X = \{ \OO(D) \in Pic(X)\ |\ D^2=\mns 2 \text{ and } D.K_X=0 \},
\]
i.e. $W_X$ is generated by the reflections $s_L \colon L' \mapsto L' - 2 \cdot \frac{D.D'}{D^2} \cdot L = L' + (D.D') \cdot L$ for $L = \OO(D) \in \Phi_X$ where $L' = \OO(D')$.
\end{itemize}
\end{prop}

\begin{proof}
Actually, the case $\rho(X)=3$ was omitted in \cite{manin}. It is clear (even without the restriction $\rho(X)\leq9$) that $W_X$ is a subset of the $K$-isometries.
The other inclusion can be shown by a direct computation for $X$ with $\rho(X)=3$.
For the rest of the proof we refer the reader to \cite[Section 26.5]{manin}.
We note that the statements in [loc.~sit.] are formulated for del Pezzo surfaces, but carry over to the case of toric surfaces verbatimly.
\end{proof}

\begin{rmk}
There are three integers associated to a toric surface $X$: the rank of $\Pic(X)$ denoted by $\rho(X)$, the number of blow-ups $l$ from a weighted projective plane, and the number of rays $n$ which equals the rank of $K_0(X)$.
Actually, they are closely related: $\rho(X)+2 = l+1 = n$.
\end{rmk}

\begin{rmk}[{\cite[Theorem 23.9(i), Section 25.5, Section 26.6]{manin}}]
In the case of $\rho(X)\leq9$, $W_X$ is finite and can be described explicitly. We recall this description for the convenience of the reader for small $\rho(X)$.
Let $(H,R_1,\ldots,R_l)$ with $l=\rho(X)-1$ be a good basis of $\Pic(X)$.
For $1\leq i<j\leq l$, we set $R_{i<j} = R_i\otimes R_j^{-1}$. Moreover, for $1\leq i<j<k\leq l$ we set
$R_{i<j<k} = H\otimes R_i\otimes R_j\otimes R_k$. 
This leads to the following table, where we list the root systems and Weyl groups for Picard ranks $3$ to $5$.
\[
\begin{array}{c|c|c|c|c}
\rho(X) & \Phi_X & \#(\Phi_X) & \# (W_X)  & \text{Type of } W_X\\
\hline
3 & \{ R_{i<j}^{\pm 1} \} & 2 & 2 & A_1\\
4 & \{ R_{i<j}^{\pm 1}, R_{i<j<k}^{\pm 1} \} & 8 & 12 & A_1 \times A_2\\
5 & \{ R_{i<j}^{\pm 1}, R_{i<j<k}^{\pm 1} \} & 20 & 120 & A_4\\
\end{array}
\]
\end{rmk}

\subsection*{Exceptional Sequences of Line Bundles}
For the general theory of derived categories in algebraic geometry, we refer the reader to \cite{huybrechts}. We denote by $\Db(X)$ the bounded derived category of coherent sheaves on a (smooth, projective, toric) surface $X$.

\begin{defn}
An object $E$ of $\Db(X)$ is called \emph{exceptional} if it fulfills
\[
\Ext^i (E,E)=
\begin{cases}
\CC & \textnormal{if $i=0$,}\\
0 & \textnormal{if $i\not=0$.}\\
\end{cases}
\]
An \emph{exceptional sequence} $\Ex$ is a finite sequence of exceptional objects $(E_1,\ldots,E_n)$
such that there are no morphisms back, that is, $\Ext^k (E_j,E_i)=0$ for $j>i$ and all $k$. 
An exceptional sequence is called \emph{full} if $E_1,\ldots,E_n$ generate $\Db(X)$, that is, the smallest full triangulated subcategory of $\Db(X)$ containing all $E_i$'s is already $\Db(X)$.
\end{defn}

\begin{rmk}
It is fairly obvious that the image of a full exceptional sequence under the projection $\Db(X) \to K_0(X)$ forms a basis of the Grothendieck group $K_0(X)$. On the other hand, it is not clear whether an exceptional sequence with $\rk K_0(X)$ elements is automatically full. Therefore, 
we call an exceptional sequence $\Ex$ \emph{$K$-full} if it  has $\rk K_0(X)$ elements.
\end{rmk}

Let $\Ex = (E_1, \ldots, E_n)$ be a $K$-full exceptional sequence of line bundles on $X$. Since tensoring with a line bundle is an autoequivalence of $\Db(X)$, we can assume without loss of generality that $E_1 = \OO_X$.
As observed in \cite{hille:perling:08}, the differences of $\Ex$ form a toric system, more precisely,
we can define a toric system $\Ts=(A_1, \ldots, A_n)$ by
\[
\OO(A_i) = E_{i+1} \otimes E_i^{-1} \ \text{ for } 1 \leq i < n \quad \text{ and } \quad A_n = - K_X-\sum_{i=1}^{n-1} A_i.
\]
We call any toric system arising this way \emph{exceptional}.
An exceptional toric system $\Ts$ is by definition $K$-full. If the corresponding exceptional sequence is full, we call $\Ts$ full as well.

\begin{rmk}
Given a toric system $\Ts=(A_1,\ldots,A_n)$ we can rearrange it cyclically in two ways.
We can \emph{rotate} $\Ts$ by shifting the index by $i$ and get another toric system $(A_{i+1},\ldots,A_n,A_1,\ldots,A_i)$. The other way is to \emph{mirror} the toric system with result $(A_n, \ldots,A_1)$. Obviously, the associated toric surface doesn't change by these operations.
Moreover, it's an easy application of Serre duality that the exceptionality of a toric system is preserved, too.
\end{rmk}

\begin{ex}[{continuation of Example \ref{ex:hirze:ts}}]
\label{ex:hirze:exseq}
It was shown in \cite{hille:perling:08} that the toric systems $\Ts_{r,i}$ are exceptional and full for all $i\in \ZZ$. The other family $\tilde\Ts_{r,i}$ is only exceptional if $r=0$ for all $i\in\ZZ$ since $P$ and $Q$ are interchangeable, or if $r=2$ and $i=0$ in which case $\tilde\Ts_{2,0} = \Ts_{2,\mns 1}$.
In these exceptional cases, $\tilde\Ts_{r,i}$ is also full.
\end{ex}

\begin{prop}[{cf. \cite[Proposition 5.5]{hille:perling:08} and \cite[Lemma 3.4]{me:ilten:11b}}]
\label{prop:aug:exfull}
Let $\Ts$ be a toric system. Then $\Ts$ is (full) exceptional if and only if $\Aug_i \Ts$ is (full) exceptional.
\end{prop}

We arrive now at the central notion of this article.

\begin{defn}
Let $\Ts$ be an exceptional toric system on $X$ of Picard rank $\rho(X)>2$. We call $\Ts$ \emph{constructible} if there is a path of toric blow-downs $X \to \cdots \to \FF_{\!r}$ to a Hirzebruch surface such that $\Ts$ can be written as successive augmentations along this path of an exceptional toric system on $\FF_{\!r}$.
\end{defn}

If $\Ts$ is constructible, then it is automatically full by Proposition~\ref{prop:aug:exfull}.
It was shown in \cite[Theorem 5.4]{me:ilten:11b}, that on a toric surface of Picard rank $3$ and $4$
any exceptional toric system is constructible. For Picard rank $5$, an example of a non-constructible toric system was given. We will investigate this counterexample more closely, but we need a further technique.

\section{Spherical Twists and Line Bundles}
\label{sec:sphtwist}

To determine the structure of a derived category $\Db(X)$, it is interesting to look at its group of autoequivalences. Some autoequivalences come for free: these are shifting the complexes, tensoring with a line bundle and pulling back by an automorphism of the underlying variety $X$.
In general, it is a difficult question whether there are additional autoequivalences. One way to construct some is by spherical objects. We refer the reader to \cite{seidel:thomas} and \cite{huybrechts} for a thorough treatment.

\begin{defn}
An object $E \in \Db(X)$ is called \emph{spherical} if $E \otimes K_X \cong E$ and 
\[
\Hom(E,E[i]) = \begin{cases} \CC & \text{for } i = 0 \text{ or } \dim X, \\ 0 & \text{otherwise.} \end{cases}
\]
For the precise definition of a \emph{spherical twist} we refer the reader to \cite[Section~8]{huybrechts}. We will only need its main property, namely,
the spherical twist $T_E$ associated to a spherical object $E$ fits into the distinguished triangle
\begin{equation}
\label{eq:twist:triangle}
\xymatrix{
\bigoplus_i \Hom(E,F[i]) \otimes E[\mns i] \ar[rr]^{ev} & & F \ar[ld] \\
& T_E(F) \ar[lu]^{[1]}
}
\end{equation}
for any object $F \in \Db(X)$.
\end{defn}

\begin{prop}[{\cite{seidel:thomas}}]
Let $T_E$ be a spherical twist.
Then $T_E$ is an autoequivalence of $\Db(X)$. Moreover, $T_E(E) \cong E[1-\dim X]$ and 
$T_E(F) \cong F$ for any $F \in \Db(X)$ with $\Hom(E,F[i]) = 0$ for all $i$.
\end{prop}

In the case of toric surfaces, spherical twists are the only autoequivalences beside the usual ones.

\begin{prop}[{\cite[Theorem 1 and 5]{broom:ploog}}]
Let $X$ be a smooth projective toric surface. Then $\Db(X)$ is generated by the automorphisms of the surface, shifting the complexes, twisting by line bundles and the spherical twists $\{ T_{\OO_C}(i)\ |\ C \text{ $(\mns 2)$-curve on } X, i \in \ZZ \}$.
\end{prop}

Let $C$ be a $(\mns 2)$-curve on a (smooth, projective) surface $X$, i.e. $C \cong \PP^1$,
$C^2 = \mns 2$ and $C.K_X = 0$. Then $\OO_C(\mns 1)$ is a spherical object in $\Db(X)$, see \cite[Example 8.10, iii)]{huybrechts}. We are interested in applying the according spherical twist to line bundles $\OO(D)$. 

\begin{prop}[{\cite[Lemma 8.12]{huybrechts}}]
\label{prop:sph:pic}
Let $X$ be a smooth projective rational surface and $C$ a $(\mns 2)$-curve on $X$.
Then the spherical twist $T_{\OO_C(\mns 1)}$ induces a $K$-isometry on $\Pic(X)$ when passing to cohomology.
\end{prop}

\begin{proof}
We present a proof very similar to the one found in [loc.~sit.], which is slightly more adapted to the specific situation.

We will now pass to cohomology, roughly along the lines drawn in \cite[Section 5.2]{huybrechts}.
First, the triangle \eqref{eq:twist:triangle} becomes  in this situation
\[
T^K_{\OO_C(\mns 1)}(\OO(D)) = \OO(D) - \chi(\OO_C(\mns 1),\OO(D)) \cdot \OO_C(\mns 1)
\]
under the projection $\Db(X) \to K_0(X)$ to the Grothendieck group.

In the next step, we apply the Chern character $\ch\colon K_0(X) \to H^\bullet(X,\QQ)$.
We note that $\ch(\OO_C(\mns 1)) = [C] \in H^2(X,\QQ)$,
where $[C]$ denotes the class of the curve $C$.
For a divisor $D$, the cohomological version of the twist is
\[
T^H_{[C]}([D]) = [D] - \chi(\OO_C(\mns 1),\OO(D)) \cdot [C].
\]
To compute the Euler pairing above, we choose a divisor $H$ on $X$ with $C.H=1$. Then $\OO_C(\mns 1) = \OO_C(\mns H)$
and there is the exact sequence
\[
0 \to \OO(-C-H) \to \OO(\mns H) \to \OO_C(\mns 1) \to 0,
\]
which induces the equality
$
\chi(\OO_C(\mns 1),\OO(D)) 
= \chi(\OO(D+H)) - \chi(\OO(D+C+H)).
$
By using Riemann-Roch and the assumptions on $C$ and $H$, we get
\begin{equation}
\label{eq:euler:cd}
\chi(\OO_C(\mns 1),\OO(D)) 
= - C.D - \frac{1}{2} C^2 - C.H - \frac{1}{2} C.K = \mns\,C.D.
\end{equation}
So the cohomological twist becomes
\[
T^H_{[C]}([D]) = [D] + C.D \cdot [C].
\]
Since $X$ is rational, $H^2(X,\ZZ)$ is isomorphic to $\Pic(X)$. Therefore,
$T^H_{[C]}$ is a $K$-isometry of $\Pic(X)$ by Proposition~\ref{prop:kiso:weyl}.
\end{proof}

\begin{rmk}
The statement can also reformulated in the following way: 
a $K$-isometry of the form $s_{\OO(C)}\colon \OO(D) \mapsto \OO(D) + (C.D)\cdot \OO(C)$ for a $(\mns 2)$-curve $C$ can be lifted to the autoequivalence $T_{\OO_C(\mns 1)}$ of $\Db(X)$.
\end{rmk}

Although the projection $T^H_{[C]}$ becomes an automorphism of $\Pic(X)$, the image of a line bundle $L$ under $T_{\OO_C(\mns 1)}$ will be a line bundle only in very specific situations.

\begin{lem}
\label{lem:twist:twoforms}
Let $X$ be a smooth, projective surface and $C$ a $(\mns 2)$-curve on $X$. Then the spherical twist $T_{\OO_C(\mns 1)}$ maps a line bundle $L = \OO(D)$ to another line bundle $L'$ if and only if one of the following two cases apply
\begin{enumerate}[(a)]
\item $D.C = 0$, in which case $L$ and $L'$ are equal, or
\label{it:twist:same}
\item $D.C = 1$, in which case there is an exact sequence
\label{it:twist:exactseq}
\[
0 \to L \to L\otimes\OO(C) \to \OO_C(\mns 1) \to 0,
\]
and $L' = L \otimes \OO(C)$.
\end{enumerate}
\end{lem}

Before we prove the statement, we note

\begin{prop}[{\cite[Theorem 6, (ii)]{broom:ploog}}]
\label{prop:thm:six}
Consider an exact triangle $E' \to E \to S$ in $\Db(X)$ for a smooth, projective surface over an algebraically closed field $k$. If $E$ is exceptional, $S$ spherical and $\Hom^\bullet(E,S) = k$, then $(E',E)$ is a so-called \emph{special exceptional pair}, i.e. $(E',E)$ form a exceptional sequence and $\Hom^\bullet(E',E) = k \oplus k[\mns 1]$.
\end{prop}

\begin{proof}[Proof of Lemma~\ref{lem:twist:twoforms}]

For the ``only if''-part, suppose
that $T_{\OO_C(\mns 1)}(L)$ is a line bundle $L'$.
Looking at the triangle in \eqref{eq:twist:triangle}, we notice that $\Hom(\OO_C(\mns 1),L)$ vanishes, since any map from a torsion to a locally free sheaf is zero.
Moreover, the kernel of the map $L \to L'$ is either $0$ or $L$ again.
Due to the form of the triangle \eqref{eq:twist:triangle} the second case cannot occur, so $L \to L'$ is injective. 
Actually, there are only two situations where this can happen. 
Either $L = L'$ and so $\Ext^i(\OO_C(\mns 1),L)$ vanishes for all $i$, 
or $\OO_C(\mns 1)$ is the cokernel of $L \hookrightarrow L'$. Then $L' = L\otimes\OO(C)$ and $\Ext^1(\OO_C(\mns 1),L)$ has to be one-dimensional and $\Ext^2(\OO_C(\mns 1),L)$ vanishes.
Using the information about the $\Ext$-groups combined with equation~\eqref{eq:euler:cd}, we deduce also the statements about the intersections.

For the ``if''-part we start with the case \eqref{it:twist:exactseq}. We note that
\[
\Hom^\bullet(L\otimes\OO(C),\OO_C(\mns 1)) = 
\Hom^\bullet(L\otimes\OO(C),\OO_C \otimes L \otimes \OO(C)) = 
H^\bullet(X,\OO_C),
\]
using for the first equality the fact that $(D+C).C=\mns 1$ with $L = \OO(D)$. So the assumptions of Proposition~\ref{prop:thm:six} are fulfilled.
As a consequence,
\[ 
\Hom^\bullet(L\otimes\OO(C),L) = H^\bullet(X,\OO(\mns C)) = 0
\]
due to the exceptionality, and due to the speciality,
\[
\Hom^\bullet(L,L\otimes\OO(C)) = H^\bullet(X,\OO(C)) = \CC \oplus \CC[\mns 1] .
\]
If we apply now $\Hom(-,L)$ to the exact sequence \eqref{it:twist:exactseq}, 
we deduce that 
\[
\Ext^i(\OO_C(\mns 1),L) \cong \Ext^{i-1}(L,L) = 
\begin{cases} 
\CC & \text{ if } i = 1,\\
0 & \text{ otherwise.}
\end{cases}
\]
So the exact sequence fits nicely in the triangle \eqref{eq:twist:triangle} of the spherical twist.

Now we turn our attention to case \eqref{it:twist:same}. Using Serre duality and $K_X.C = 0 = L.C$
we deduce
\begin{multline*}
\Hom^\bullet(\OO_C(\mns 1),L) \cong 
\Hom^{n-\bullet}(L,\OO_C(\mns 1) \otimes K_X) = \\
= H^{n-\bullet}(X,\OO_C(\mns 1) \otimes K_X \otimes L ) =
H^{n-\bullet}(C,\OO_C(\mns 1)) = 0.
\end{multline*}
So the left corner of the triangle \eqref{eq:twist:triangle} is zero and $T_{\OO_C(\mns 1)}(L)=L$.
\end{proof}

\begin{rmk}
Actually, the case \eqref{it:twist:same} of Lemma~\ref{lem:twist:twoforms} is an instance of the general property $T_E(F) = F$ for any spherical object $E \in \Db(X)$ and $F \in E^\perp$.
Note that in case \eqref{it:twist:exactseq} with $T_{\OO_C(\mns 1)} L = L\otimes\OO(C)$, twisting
$L\otimes\OO(C)$ will \emph{not} result in another line bundle, since $(D+C).C = \mns 1$ for $L = \OO(D)$.
\end{rmk}

\section{Non-Constructible Exceptional Sequences}
\label{sec:nonconstr:exseq}

In \cite[Section 5.4]{me:ilten:11b}, an example was given of an exceptional toric system $\Ts$ which was \emph{not} constructible.
There, a subtle detail was omitted, namely, whether the corresponding exceptional sequence is \emph{full}. 
Using spherical twists, we can address this question.

We will concentrate on the toric surface $X=\TV(\mns 2,\mns 1,\mns 1,\mns 1,\mns 1,\mns 2,\mns 1)$.
First, we fix a good basis $(H,R_1,R_2,R_3,R_4)$ of $\Pic(X)$. For that we choose a way of blowing down $X$ to a Hirzebruch surface, say, as in Figure~\ref{fig:blowdown:X}.
\begin{figure}
\figBlowdownX
\caption{Blowing down $X = \TV(\protect\mns 2,\protect\mns 1,\protect\mns 1,\protect\mns 1,\protect\mns 1,\protect\mns 2,\protect\mns 1)$ to $\FF_1$.}
\label{fig:blowdown:X}
\end{figure}
Using this way, we get the good basis
\[
H = D_1+D_2+D_3+D_7,\ R_1 = D_3,\ R_2 = D_1+D_7,\ R_3 = D_5,\ R_4 = D_7.
\]
We can express the standard exceptional toric system $\Ts_X=(D_1,\ldots,D_7)$ on $X$ with respect to this basis and get
\[
\Ts_X
\ =\ 
\begin{pmatrix}
H & R_1 & R_2 & R_3 & R_4
\end{pmatrix}
\cdot
\scalebox{0.8}{$
\begin{pmatrix}
 0 &  1 &  0 &  1 &  0 &  1 &  0 \\ 
 0 & \mns 1 &  1 & \mns 1 &  0 &  0 &  0 \\ 
 1 & \mns 1 &  0 &  0 &  0 & \mns 1 &  0 \\ 
 0 &  0 &  0 & \mns 1 &  1 & \mns 1 &  0 \\ 
\mns 1 &  0 &  0 &  0 &  0 & \mns 1 &  1 \\ 
\end{pmatrix}$} \text{ . }
\]
By applying all $K$-isometries of $\Pic(X)$ to $\Ts$, we obtain $120$ toric systems.
It turns out that $98$ of them are actually exceptional. If we choose as a basis of $\Pic(X)$ the exceptional divisors $D_2, D_3, D_4, D_5$ and $D_7$, we can check which toric systems are constructible in the following way.

Since every exceptional toric system is constructible on a toric surface of Picard rank $4$ due to \cite[Theorem 5.4]{me:ilten:11b}, it is sufficient to determine whether the exceptional toric systems are augmented once. 
So we need to check whether the toric system $\Ts$ contains at least one exceptional divisor $R$ at, say position $i$, such that ``de-augmentation'' $(A_1, \ldots, A_{i-2}, A_{i-1}+R, A_{i+1}+R, A_{i+2}, \ldots, A_n)$ is contained in $R^\perp$, see Remark~\ref{rmk:exdiv:orth} and Definition~\ref{defn:augment}.

Checking which of these toric systems are exceptional or augmented was done using the {\tt ToricVectorBundles} package for Macaulay~2, see \cite{tvb}.
The corresponding Macaulay~2 script is available online at:
\begin{quote}
{\url{http://page.mi.fu-berlin.de/hochen/exseq-twist.m2}}
\end{quote}
The result is that almost all exceptional toric systems are constructible, but there are two exceptions
\[
\Ts = 
\begin{pmatrix}
D_2 & D_3 & D_4 & D_5 & D_7
\end{pmatrix}
\cdot
\scalebox{0.8}{$
\begin{pmatrix}
0  &1 & \mns 1 &1  &0 & 0 & 0  \\
\mns 1 &1 & 0  &1  &0  &0  &1  \\
0  &0  &1  &0  &1  &\mns 1 &1  \\
0  &0  &1  &\mns 1 &1  &0  &0  \\
1  &\mns 1 &0  &0  &\mns 1 &1  &\mns 1 \\
\end{pmatrix}$}
\]
and $f^\ast \Ts$, where $f$ is the automorphism of $X$ depicted in Figure~\ref{fig:auto:x} and where we apply the pull-back component-wise. Since $f^\ast$ itself induces a $K$-isometry on $\Pic(X)$, we can restrict our attention to $\Ts$.
\begin{figure}
\figAutoX
\caption{An automorphism $f$ of $X= \TV(\protect\mns 2,\protect\mns 1,\protect\mns 1,\protect\mns 1,\protect\mns 1,\protect\mns 2,\protect\mns 1)$.}
\label{fig:auto:x}
\end{figure}
The exceptional sequence corresponding to $\Ts$ is
\[
\Ex =
\begin{pmatrix}
D_2 & D_3 & D_4 & D_5 & D_7
\end{pmatrix}
\cdot
\scalebox{0.8}{$
\begin{pmatrix}
0& 0&  1& 0& 1& 1&  1 \\
0& \mns 1& 0& 0& 1& 1&  1 \\
0& 0&  0& 1& 1& 2&  1 \\
0& 0&  0& 1& 0& 1&  1 \\
0& 1&  0& 0& 0& \mns 1& 0 \\
\end{pmatrix}$} \text{ . }
\]
There are two $(\mns 2)$-curves on $X$, namely, $D_1$ and $D_6$. Since $f^\ast(\OO(D_1)) = \OO(D_6)$ we can restrict to the spherical twist $T = T_{\OO_{D_1}(\mns 1)}$. Actually, this twist fulfills the assumptions of Lemma~\ref{lem:twist:twoforms} and maps the exceptional sequence $\Ex$ to the constructible exceptional sequence
\[
T (\Ex) = 
\begin{pmatrix}
D_2 & D_3 & D_4 & D_5 & D_7
\end{pmatrix}
\cdot
\scalebox{0.8}{$
\begin{pmatrix}
0& 1&  1& 1&  1& 1&  2  \\
0& 0&  0& 1&  1& 1&  2  \\
0& 0&  0& 1&  1& 2&  1  \\
0& \mns 1& 0& 0&  0& 1&  0  \\
0& 0&  0& \mns 1& 0& \mns 1& \mns 1 \\
\end{pmatrix}$} \text{ . }
\]
So we have shown, that $\Ex$ and $f^\ast \Ex$ are \emph{full} non-constructible exceptional sequences on $X$.
Further computational evidence has led to the formulation of the following

\begin{conj}
Given an exceptional sequence $\Ex$ of line bundles on a toric surface $X$ of length $n = \rk K_0(X)$. 
Then $X$ is either constructible or differs by a finite sequence of spherical twists $T_{\OO_C(\mns 1)}$ from a constructible one.
Consequently, $\Ex$ is automatically \emph{full}.
\end{conj}

\bibliographystyle{alpha}
\bibliography{sph-twist}

\end{document}